\documentclass[fleqn,a4paper,11pt]{article}
\usepackage{amsmath,amsfonts,calrsfs,amssymb,color}
\usepackage{amsthm,dsfont}

\newtheorem{Theorem}{Theorem}[section]

\newtheorem{Proposition}[Theorem]{Proposition}
\newtheorem{Lemma}[Theorem]{Lemma}
\newtheorem{Corollary}[Theorem]{Corollary}
\newtheorem{Remark}[Theorem]{Remark}
 
\newtheorem{Hypothesis}{Hypothesis}  
\makeatletter
\@addtoreset{equation}{section}

\makeatother
 \newcommand{\U}{\mathbb U}
\newcommand{\R}{\mathbb R}
\newcommand{\N}{\mathbb N}

\newcommand{\E}{\mathbb E}

\renewcommand{\P}{\mathbb P}
\newcommand{\Q}{\mathbb Q}
\newcommand{\A}{\mathbb A}
\newcommand{\X}{\mathbb X}
\newcommand{\Y}{\mathbb Y}
\def\ds{\displaystyle}

\def\blu{}

\def\nor{\color{black}}

\title{\bf  A mild Girsanov formula}
 
\newlength{\tempa}
\divide\tempa by 2
\newlength{\tempb}
 
 \makeatletter
\newcommand{\crossout}[1]{%
  \begingroup
  \sbox\z@{#1}%
  \dimen\z@=\wd\z@
  \dimen\tw@=\ht\z@
  \dimen\z@=.99626\dimen\z@   
  \dimen\tw@=.99626\dimen\tw@ 
  \edef\co@wd{\strip@pt\dimen\z@}
  \edef\co@ht{\strip@pt\dimen\tw@}
  \leavevmode
  \rlap{\pdfliteral{q 1 J 0.4 w 0 0 m \co@wd\space \co@ht\space l S Q}}%
  \rlap{\pdfliteral{q 1 J 0.4 w 0 \co@ht\space m \co@wd\space 0 l S Q}}%
  #1%
  \endgroup
}
\makeatother

\author{Giuseppe Da Prato\\
\small Scuola Normale Superiore, Pisa, Italy
\and  
\and Enrico Priola\\
\small University of Pavia, Pavia, Italy
\and
Luciano Tubaro\footnote{correspondig author}\\
\small University of Trento, Trento, Italy
}

\date{July 31, 2025 }

\begin{document} 
 \maketitle
\section*{} \hspace{4cm} \vbox{\hsize=11cm\it\footnotesize\noindent  Dedicated to the memory of our mentor and friend Giuseppe Da Prato, who was still thinking on this work until the day before his death (October 5, 2023).}
\  
  
 \begin{abstract}
 We consider the  SPDE$\colon dZ=(AZ+b(Z)) dt+dW(t),\,Z_0=x$,
 on a separable Hilbert space $H$, where $A\colon H\to H$ is self-adjoint $b\colon H\to H$ is  Lipschitz   continuous and $W$ is a cylindrical Wiener process on $H$. We determine, with the help of  a well-known  formula for nonlinear transformations  of Gaussian integrals due to R. Ramer \cite{Ra74}, an explicit representation  for the law of $Z_x$  in $C([0,T];H)$, see Theorem \ref{t3.1} below.

When $b$ is,  in addition, dissipative,  we determine  the invariant measure  $\nu$ of the  semigroup $P_t\varphi(x)=\E[\varphi(Z_x(t))]$,  the corresponding stationary process $Z_{\R}$. 
The final Section 5 is devoted to colored noise.

 \end{abstract}
\noindent {\bf 2000 Mathematics Subject Classification AMS}: 60H15, 60H07, 60H30.

\noindent {\bf Key words}:  Girsanov formula, Stochastic mild equations, Malliavin Calculus, Skorokhod integral, Stationary process.

 \newpage
 
 \tableofcontents

\newpage
\section{Introduction and setting of the problem}

Let  $H$ be a real separable Hilbert space, (norm $|\cdot|_H$, scalar product $\langle \cdot, \cdot\rangle_H$).
We are concerned with the following  stochastic differential equation on $H$,
\begin{equation}
\label{e1.1}
\left\{\begin{array}{l}
dZ_x(t)=(AZ_x(t)+b(Z_x(t))dt+dW(t),\quad t\ge 0,\\
\\
Z_x(0)=x\in H,
\end{array}\right.
\end{equation}
under the  following  assumptions.
\begin{Hypothesis}
\label{h1}
(i) $A\colon D(A)\subset H\to H$ is  a self--adjoint  operator  and there is $\omega>0$ such that
$\|e^{tA}\|_{\mathcal L (H)}\le e^{-\omega t},\, t\ge 0.$
Moreover,  $(-A)^{-1+\beta}$ is of trace class for some $\beta\in\, ]0,1[$.\medskip

\noindent(ii) $b\colon H\to H$ is  Lipschitz  continuous. 
\medskip


\noindent(iii) $W$ is an $H$--valued cylindrical Wiener process on a filtered probability space
$(\Omega, \mathcal F, (\mathcal F_t)_{t\ge 0},\P)$.
 \end{Hypothesis}

\bigskip
 
\noindent We shall denote by   $W_A$ the {\em stochastic convolution}
\begin{equation}
\label{e1.3}
  W_A(t)=\int_0^te^{(t-s)A}\, dW(s),\quad t\ge 0.
\end{equation}
Thanks to Hypothesis \ref{h1}(i), $W_A$  is a continuous process, see, e.g., \cite[Theorem 2.6]{Da04}.

We fix   $T>0$ and set $E_{[0,T]}=C([0,T];H)$, the space of all continuous mapping $[0,T]\to H$ endowed with the sup norm; we shall denote by $\mathcal B(E_{[0,T]})$  the $\sigma$--algebra of all Borel subsets of $E_{[0,T]}$. Moreover,  $ B_b(E_{[0,T]})$ will represent the space of all mappings $E_{[0,T]}\to \R$  which are Borel and bounded. 
 
The law of $W_A$ in $E_{[0,T]}$ is Gaussian $\mathcal N_{\Q_T}$, that is, it  has mean $0$ and covariance $\Q_T$;  it will be described in Section 2 below.
 \bigskip

The  goal of this paper is to represent the 
law of $Z_x$ on $E_{[0,T]}$, for all $x\in H$,  as a suitable integral with respect to  $\mathcal N_{\Q_T}$, that we shall call a {\em mild Girsanov} formula, and to  deduce some consequences from this identity.\bigskip

We shall proceed as follows. Setting $Z_x(t)-e^{tA}x=:K_x(t),\, t\in[0,T],$ equation \eqref{e1.1} becomes
\begin{equation}
\label{e1.4}
K_x(t)=\int_0^te^{(t-s)A}\,b(K_x(s)+e^{sA}x)ds+ W_A(t),\quad t\in[0,T].
\end{equation}
Then we associate with  the stochastic equation \eqref{e1.4}  the following deterministic one,
\begin{equation}
\label{e1.5}
 k_x(t)=\int_0^te^{(t-s)A}b(k_x(s)+e^{sA}x)ds+h(t),\quad h\in E_{[0,T]},\,t\in [0,T],
\end{equation}
that, by a  standard fixed point argument,  has a unique   solution in $E_{[0,T]}$  that we denote by $k_x(t),\,t\in [0,T]$. 
Next, for  every $x\in H$  we consider     the mapping 
\begin{equation}
\label{e1.6}
F_x\colon E_{[0,T]}\to E_{[0,T]},\quad   h\to F_x(h)=k_x.
\end{equation}
We shall denote by $G_x: E_{[0,T]}\to E_{[0,T]}$ the inverse mapping  of $F_x$, which reads as follows
\begin{equation}
\label{e1.7}
[G_x(h)](t)=h(t)-\int_0^t e^{(t-s)A}\,b(h(s)+e^{sA}x)\,ds,\quad  h  \nor \in E_{[0,T]}. 
\end{equation} 
Since $b$ is   Lipschitz continuous,  both $F_x$  and $G_x$
are continuous; therefore they are homeomorphisms  from $E_{[0,T]}$ onto $E_{[0,T]}$.  

A key role will  be played  in what  follows by the mapping
$\gamma_x(h)=h-G_x(h),\,h\in E_{[0,T]},$ (see e.g.,  identity \eqref{e1.10}),
\begin{equation}
\label{e1.8}
[\gamma_x(k)](t)=\int_0^t e^{(t-s)A}b(k(s)+e^{sA}x)\,ds,\quad k\in E_{[0,T]},\,t\in[0,T].
\end{equation} 
Now the solution $K_x(\cdot)$ of \eqref{e1.4} can be written as $K_x(\cdot)=F_x( W_A(\cdot)).$
So, the law of $Z_x$ on $E_{[0,T]}$ is given  by
$$
(\P\circ Z_x^{-1})(\Phi)=\E[\Phi(Z_x(\cdot))]=\E[\Phi(K_x(\cdot)+e^{\cdot A}x)]= \E[\Phi(F_x( W_A(\cdot))+e^{\cdot A}x)],\quad\forall\,\Phi\in B_b(E_{[0,T]}).
$$
Therefore,  by a change of  variables, it results
\begin{equation}
\label{e1.9}
(\P\circ Z_x^{-1})(\Phi)= 
\int_{E_{[0,T]}}\Phi(F_x( h(\cdot))+e^{\cdot A}x)\,\mathcal N_{\Q_T}(dh),\quad \forall\,\Phi\in  B_b(E_{[0,T]}).
\end{equation}
This  formula is usual when dealing with stochastic equations with additive noise, see e.g., \cite{FrVe12}.
But we look for a formula involving $G_x$ rather than $F_x$ because $G_x$ is explicit while $F_x$ is not.

So, we set
$F_x(h)=k,$
and obtain
\begin{equation}
\label{e1.10}
(\P\circ Z_x^{-1})(\Phi)= 
\int_{E_{[0,T]}}\Phi(k+e^{\cdot A}x)\,(\mathcal N_{\Q_T}\circ G_x)(dk),\quad \forall\,\Phi\in  B_b(E_{[0,T]}).
\end{equation}
But now we need to show that
\begin{equation}
\label{e1.10e}
N_{\Q_T}\circ G_x \ll N_{\mathcal \Q_T}.
\end{equation}
In fact, using   an identity due to  
R. Ramer, \cite{Ra74} (see Section 3.1 below) we show that
 \begin{equation}
 \label{ei1.11}
\varrho(x,h):=\frac{d(\mathcal N_{\Q_T}\circ G_x)}{d\mathcal N_{\Q_T}}\,(h)=\exp\left\{ -\tfrac12|\gamma_x(h)|^2_{\mathcal H_{\Q_T}}  + [I(\gamma_x)](h)\right\},\quad \mathcal N_{\Q_T}-a.s.
\end{equation}
where $\mathcal H_{\Q_T}$ denotes the Cameron--Martin space of $\mathcal N_{\Q_T}$ and $I(\gamma_x)$ the It\^o integral of $\gamma_x$, which results to be adapted.

So, we end up with, the identity
\begin{equation}
\label{e1.12}
(\P\circ Z_x^{-1})(\Phi) =\int_{E_{[0,T]}}\Phi(h+e^{\cdot A}x)\, \exp\left\{ -\tfrac12|\gamma_x(h)|^2_{\mathcal H_{\Q_T}}  + [I(\gamma_x)](h)\right\}\mathcal N_{\Q_T}(dh), \end{equation}
which  will be proved in Theorem \ref{t3.1} below.

As a consequence   of  \eqref{e1.12} we obtain  an expression for the transition semigroup  $P_T,\,T\ge 0,$ corresponding to the process $Z_x$, see Corollary \ref{c3.3},
\begin{equation}
\label{e1.14}
P_T\varphi(x)= \E[\varphi(Z_x(T))]=
\int_{E_{[0,T]}}\varphi(h(T)+e^{TA}x)\, \varrho(x,h)\;\mathcal N_{\Q_T}(dh).
\end{equation} 
\begin{Remark}
\em When $A=0$ identity \eqref{e1.12} reduces to  the  classical Girsanov formula, see e.g., \cite{Bo98}.
\end{Remark}
   
We end this section with  some comments about identities \eqref{e1.12} and  \eqref{e1.14}.
 First we note  that
the  classical Girsanov formula describes the   law of $Z_x$ in terms of the Wiener measure on $C([0,T];H)$,  and so,    does not provide any    information about   the   asymptotic behaviour
of $Z_x$. Instead the mild Girsanov formula allows \nor to guess the explicit form of the invariant measure  $\nu$ of $P_t,\,t\ge 0$ in case this measure  exists and of the invariant measure $\mu$ of the stochastic convolution (stationary O.U). 
\bigskip
 
For   instance when $P_t$ is strongly mixing or when a  suitable coupling argument  is available for equation \eqref{e1.1} one  can prove  that  the limit below exists
$$
\lim_{T\to \infty}\,P_T\varphi(x)=\int_H \varphi(y)\nu(dy),\quad \forall\, x\in H,\, \varphi\in C_b(H).
$$ 
Then it remains to pass to limit  as $T\to\infty$ on the right hand side of \eqref{e1.14}. This is not obvious in general, however, and it will be the object of future  research.
In this paper, we limit ourselves to  considering  the case where $b$ is  dissipative.
 In this case first, we describe the law  of   the stationary  process $Z_{\R}$ on the Fr\'echet space $C(\R; H)$ (the space where  $Z_{\R}$ lives; cf. \eqref{e5.24}), Theorem \ref{t4.3}. Then, taking advantage of  this result,  we provide explicit representation formulae both for $\nu, \mu$ and for the density $\tfrac{d\nu}{d\mu}$, Theorem \ref{t3.3z} and Proposition \ref{p4.6}.  

 \medskip

Let us explain briefly the content of the paper.

The application of the Ramer formula requires some preliminaries as: the Cameron--Martin space of $\mathcal N_{\Q_T}$, determined in 
Section 2  as well the Malliavin derivative and the Gaussian divergence operator.  The mild Girsanov formula, Theorem \ref{t3.1}, is proven in Section 3.

 Section 4 is devoted to the case when $b$ is dissipative.
Section 4.1 to some  preliminaries on  Gaussian measures on locally 
 convex spaces, see
the monograph  from  V. I. Bogachev, \cite{Bo98}. 

  Section 4.2 is devoted to an explicit formula for the  law of the stationary process  $Z_{\R}$   in the   space
 $C(\R;H)$. Section  4.3 to invariant measure $\nu$ and $\mu$ and to their relations.

In Section 5 we show that all result of the paper can be easily generalised to SPDEs with an additive colored noise. 

Finally, Appendix A is devoted to recall some maximal regularity results for abstract evolution equations,
which are needed in what follows.

 \section{The Cameron--Martin space }

 Since
 $$  
 \E\int_0^T|W_A(t)|^2_H\,dt=\frac12 \int_0^T \mbox{\rm Tr}\,[(-A)^{-1}\,(1-e^{2tA})] dt<\infty, 
  $$
    we can extend the measure $\mathcal N_{\Q_T}$ from $E_{[0,T]}$ to $ \X:=L^2(0,T;H)$  
 which we shall	 denote by  $\mathcal N_{\overline{\Q}_T}$. 
 It  is well known that $ \mathcal N_{\overline{\Q}_T}$   is Gaussian, with mean $0$ and covariance $\overline{\Q}_T$ given by
\begin{equation}
\label{e2.1}
(\overline{\Q}_T \,h)(t)=\int_{0}^{T} K(t,s)h(s)\, ds, \quad t \in [0,T],\;h\in \X,
\end{equation}
where  
\begin{equation}
\label{e2.2}
K(t,s)=
\int_{0}^{\min\{t,s\}}\,e^{(t+s-2r)A} dr \quad t,s\in[0,T],
\end{equation}
see e.g., \cite[Theorem 5.4]{DaZa14}.
The measure  $\mathcal N_{\overline{\Q}_T }$
 is concentrated on $E_{[0,T]}$ and its Cameron--Martin space  $\mathcal H_{\overline{\Q}_T}=\overline{\Q}_T ^{1/2}(\X)$ coincides with that of  
 $\mathcal N_{\Q_T }$, see e.g., \cite[Proposition 2.10]{DaZa14}.
   
The following lemma is  well  known;  we give a sketch  of its proof, however, for the reader convenience.  
\begin{Lemma}
\label{l2.1}
Setting $\A_T=-\overline{\Q}_T^{-1}$, \nor it results
\begin{equation}
 \label{e2.3}
\left\{\begin{array}{l}
  \A_T  f=f^{\prime\prime}(t)-A^{2}\,f(t),\quad \forall\,f\in D(\A_T) \\\\
D(\A_T)=\{f\in L^2(0,T;D(A^2))\cap W^{2,2}(0,T;H):\,\;f(0)=0,\;f^\prime(T)=Af(T)\}.
\end{array}\right.
 \end{equation}
\end{Lemma} 
\begin{proof}  
 Given $h\in \X$ set $\overline{\Q}_T\, h=f$ and write 
 $$
\int_{0}^tK(t,s)\,h(s)\,ds+\int_{t}^TK(t,s)\,h(s)ds =f(t),\quad t\in[0,T].
 $$
 So,
\begin{equation}
\label{e2.4}
\int_{0}^t\left( \int_{0}^se^{(t+s-2r)A}\,dr \right)\,h(s)\,ds+\int_{t}^T\left(  \int_{0}^te^{(t+s-2r)A}\, dr \right)\,h(s)ds =f(t). 
\end{equation}
Note that $f(0)=0.$   Introducing $Q_t = \int_0^t e^{2sA} ds = (-2 A)^{-1}(I-e^{2tA}),$ $ t \ge 0,$  and
  differentiating \eqref{e2.4} with respect to 
$t$ yields
\begin{displaymath}
\begin{array}{l}
\ds Q_t\,h(t)+ A \int_{0}^t\left( \int_{0}^se^{(t+s-2r)A}\,dr \right)\,h(s)\,ds\\
\\
\ds-Q_t\,h(t)+\int_{t}^T e^{(s-
t)A}\,h(s)ds+A\int_t^T\left( \int_{0}^te^{(t+s-2r)A}\,dr\right)\,h(s)\,ds  =f^\prime(t).
\end{array}
\end{displaymath}
It follows that
\begin{equation}
\label{e2.5}
Af(t)+\int_{t}^T e^{(s-
t)A}\,\,h(s)ds =f^\prime(t),
\end{equation}
which implies
$f^\prime(T)= Af(T).$
 Now, differentiating   \eqref{e2.5} with respect to $t$, a.e.,  yields
 \begin{equation}
\label{e2.6}
Af^\prime(t)-h(t) -A\int_{t}^T   e^{(s-
t)A} \,h(s)ds =f^{\prime\prime}(t).
\end{equation}
Taking into account \eqref{e2.5}, we deduce
 \begin{equation}
\label{e2.7}
f^{\prime\prime}(t) -A^2f(t)=-h(t).
\end{equation}
Since $h=-\A_T f$, we have
$$
\A_T f=f^{\prime\prime}(t)-A^2f(t),
$$
as required.
 \end{proof}
 
 Now we can determine  the {\em Cameron--Martin space}  of $\mathcal N_{\overline{\Q}_T}$ which, as remarked earlier, coincides   with that of  $\mathcal N_{\Q_T}$.
\begin{Proposition}\label{p2.2} 
 It results, see Appendix A, \nor  
 $$
 \mathcal  H_{ \Q_T \nor }  =\mathcal H_{\overline{\Q}_T}=\{u\in L^2(0,T;D(A))\cap  W^{1,2}(0,T;H):\, u(0)=0\}:=\Gamma_A(H).
 $$
 Moreover, if $u\in \mathcal H_{\overline{\Q}_T}$, we have
\begin{equation} 
\label{e2.8}
 |u|^2_{\mathcal H_{\overline{\Q}_T}}:= |\overline{\Q}_T^{\,-1/2}u|^2_{\X } =|(-A)^{1/2}\,u(T)|_H^2 +\int_0^T (| u^\prime(t)|_H^2+| Au(t)|_H^2)\,dt.
\end{equation}
\end{Proposition}

\begin{Remark}
\label{r2.3}
\em
Note that by interpolation it results, see e.g., \cite{Lu18},
 \begin{equation}
\label{e2.9}
\{u\in L^2(0,T;D(A))\cap  W^{1,2}(0,T;H):\, u(0)=0\} \subset C([0,T];D(-A)^{1/2})),
\end{equation}
so that  the term $|(-A)^{1/2}\,u(T)|_ H$ in \eqref{e2.8} is meaningful.
\end{Remark}
{\it Proof of Proposition \ref{p2.2}}.
 It is enough to assume $f\in D(\A_T)$, since
$D(\A_T)$ is dense in $D((-\A_T)^{1/2})$. In this case  we can write
$$
\begin{array}{l}
\ds\int_0^T\langle\A_T f(t),f(t)\rangle _Hdt=\int_0^T \langle \,f^{\prime\prime}(t),f(t)\rangle_H\, dt-
\int_0^T \langle   A^{2} f(t),f(t)\rangle_H\, dt\\
\\
\ds= \langle   f^\prime(T),f(T)\rangle _H-\int_0^T |\,f^\prime(t)|_H^2\,dt
-\int_0^T|Af(t)|_H^2\,dt\\
\\
\ds=-|(-A)^{1/2}\,f(T)|_H^2 -\int_0^T | f^\prime(t)|_H^2\,dt
-\int_0^T|A\,f(t)|_H^2\,dt.
\end{array}
$$
Then \eqref{e2.8} follows.
\hfill $\Box$

   We  now prove a result which  provides a useful information
 on the support of $\mathcal N_{\overline{\Q}_T}$.
 \begin{Lemma}
\label{l2.4}
Assume Hypothesis \ref{h1}.
  Then it results 
\begin{equation}
\label{e2.10}
\int_{\X}|(-A)^{\beta/2} h|^2_{\X}\,\mathcal N_{\overline{\Q}_T}(dh)\le \tfrac{T}2\,\mbox{\rm Tr}\,[(-A)^{\beta-1}]<\infty.
\end{equation}
\end{Lemma}
\begin{proof}
Write
$$
\begin{array}{l}
\ds\int_{\X}|(-A)^{\beta/2} h|_{ \X  }^2\,\mathcal N_{\overline{\Q}_T}(dh)
=\E\int_0^T|(-A)^{\beta/2}\,W_A(t)|_H^2\,dt\\
\\
\ds=\E\int_0^Tdt\left|\int_0^t(-A)^{\beta /2}e^{(t-s)A}dW(s)\right|_H^2\\
\\
\ds=\int_0^Tdt\int_0^t\,\mbox{\rm Tr}\,[(-A)^{\beta }\,e^{2(t-s)A}]ds\\
\\
\ds=\tfrac12\int_0^T\mbox{\rm Tr}\,[(-A)^{\beta-1}\,(I-e^{2tA})]\,dt\le \tfrac{T}2\,\mbox{\rm Tr}\,[(-A)^{\beta  -1}].
\end{array}
$$ 
The conclusion follows.
\end{proof}  
 \begin{Remark}
 \label{r2.5}
 \em By Lemma  \ref{l2.4} it follows that $\mathcal N_{\overline{\Q}_T}$
 is concentrated on $L^2(0,T;D((-A)^{\beta/2}))$,
 where $\beta$  is given in Hypothesis \ref{h1}. \nor 
  \end{Remark} 
  
  \subsection{Brownian Motion}
   We are going to define
   the Brownian motion on  $\X$ and $E_{[0,T]}$.
   To this purpose we first introduce the {\em white noise} function.  Let    $(\psi_j)$ be an orthonormal basis in $\X$ of eigenfunctions of $\overline{\Q}_{T}$ and 
let  $(\lambda_j)$ be  the sequence of the
corresponding eigenvalues,
$$
\overline{\Q}_{T}\psi_j=\lambda_j\psi_j,\quad j\in\N.
$$
For all $f\in \X$ we set
 $$
W_f(h)=\langle \overline{\Q}_{T}^{-1/2}h,f\rangle_{\X}=\sum_{j=1}^\infty\lambda_j^{-1/2}\langle  h,\psi_j \rangle_{\X}
\langle  f,\psi_j \rangle_{\X},\quad h\in  \X,\; \mathcal N_{\overline{\Q}_T}\mbox{-a.e.}  
$$
Finally, we    define a sequence of   Brownian motions in $\X$ by proceeding as in \cite{Da13}
 \begin{equation}
\label{e211}
B^\alpha(t)=[W_{{\mathds 1}_{[0,t]}e^\alpha}],\quad t\in[0,T],\,\alpha\in\N,\quad B(t)=\sum_\alpha B^\alpha(t)\,e_\alpha
\end{equation}
where $(e_\alpha)$ is an orthonormal basis on $H$.
As proved in \cite{Da13}, $B^\alpha$ is continuous
for all $\alpha\in\N$.
 
\vskip 1mm 
 Recall that $\gamma_x \in L^2(E_{[0,T]} \times [0,T];H ) \sim L^2 (E_{[0,T]}; \X)$\footnote{We put the product measure
 $ \mathcal N_{\overline{\Q}_T} \times dt$ on $E_{[0,T]} \times [0,T]$}\nor 
 
 The following result is standard. Let $\mathcal M$  be the Malliavin derivative.    We only point out that ${\mathcal M}^*(\gamma_x)= -\delta (\gamma_x)$ according to  the notation of \cite{Bo98} and \cite{Nu06}. \nor
\begin{Proposition}
\label{p10}
Let $x\in H$ then $\gamma_x\in D(\mathcal  M^*)$ and  it results  for any $h \in E_{[0,T]}$, $\mathcal N_{\overline{\Q}_T}$-a.e.,
 \begin{equation}
\label{e19}
[\mathcal  M^*(\gamma_x)](h)=\int_0^T\langle[\gamma_x(h)](s),d[B(s)](h)\rangle_H=\int_0^T \int_0^s\langle  e^{(s-r)A} \nor  b(h(r)+e^{rA}x)  \,dr,d[B(s)](h)\rangle_H.
\end{equation}
Moreover 
 \begin{equation}
\label{e18aa} \blu  
\int_\X|[\mathcal  M^*\gamma_x](h)|^2\,\mathcal N_{\overline{ \Q}_T }(dh)=\int_\X\int_0^T|[\gamma_x(h)](s)|_H^2\,ds\,\mathcal N_{\overline{ \Q}_T}(dh). \nor
\end{equation}
\end{Proposition}
\begin{proof}
First note that $\gamma_x$ is adapted to the natural filtration, so that the conclusion follows from  the standard properties of the It\^o integral.

\end{proof}
In the following we shall write sometimes for brevity
 $$
 \mathcal  M^*(\gamma_x)=  \int_0^T   \langle [\gamma_x(\cdot)](s),dB(s)\rangle_H. \nor  
$$

  \section{Proof of the mild Girsanov formula}
   \subsection{The Ramer identity}
   
 Here we recall  the Ramer identity  (or Kusuoka-Ramer theorem, \cite{Ku82}). 
   Such theorem has been also applied to study stochastic boundary value problems (see \cite{NuaPar}, \cite{CapPri} and the references therein). 
  
 The following  version is an extension to complete  separable locally convex spaces (in particular to  separable Fr\'echet spaces) of a result  given by  D. Nualart \cite[pag. 240]{Nu06} (see also \cite[Lemma 4.1.2]{Nu06} and Section 6.6 in Bogachev \cite{Bo98} for an analogous result in general locally convex spaces).  In this section we apply  this identity in a Banach space of continuous functions, but we will need such extension in Section 4.
  
 \blu We write $\phi \in D(\mathcal M)$ to indicate that $\phi$ is differentiable along the directions of the Cameron-Martin space  or Malliavin differentiable (cf. \cite{Bo98} and \cite{Nu06}). Moreover $\mathcal M^* = -\delta$ \nor
 
  \begin{Proposition}
\label{p3.1}
Let $\Y$ be a  complete   separable locally convex space and $\Lambda$ a homeomorphism of $\Y$ onto $\Y$. Let $ \mathcal N_\U$ be a   Gaussian measure  on $\Y$ of mean $0$, covariance $\U$ and Cameron--Martin space $\mathcal H_\U$. Set $\phi(h)=h-\Lambda^{-1}h$, for all  $h\in \Y$, and assume that
 $\phi\in D(\mathcal M)$ and $\phi(h)$ belongs to $\mathcal  H_{\mathcal \U}$, for  $h \in \Y$, $ \mathcal N_\U$-a.e.,  where $\mathcal M$ is  the   the Malliavin derivative  
in $\Y$. Suppose that ${\det}_2(I-\mathcal M\phi(h)) \not =0$, $ \mathcal N_\U$-a.e.. 
 
  \medskip
Then it results $\mathcal N_{\U}\circ \Lambda^{-1}\ll \mathcal N_\U$ and
 \begin{equation}
\label{e3.1}
\frac{d\mathcal N_{\U}\circ \Lambda^{-1}}{d\mathcal N_\U}(h) =\exp\left\{-\tfrac12|\phi(h)|^2_{\mathcal H_{\U }} + \mathcal M^*\phi(h)\right\}\,{\det}_2(I-\mathcal M\phi(h)),
\end{equation} 
where 
 the determinant \nor  is intended  in the sense of  Carleman--Fredholm, see $\cite{GoKr69}$.
\end{Proposition}
Now we are ready to show the main result of this section. 
  \begin{Theorem}
  \label{t3.1}
  Assume  Hypothesis \ref{h1}, let $x\in H,\, h\in E_{[0,T]}$
and set 
  \begin{equation}
  \label{e3.2}
 [\gamma_x(h)](t)=\int_0^te^{(t-s)A}b(h(s)+e^{sA}x)\,ds,\quad  \,t\in[0,T].
\end{equation}
  Then, \medskip
   
 \noindent (i)    $\gamma_x(h)$ belongs to  $\mathcal H_{\Q_T}$ for     $h\in E_{[0,T]}, \,\mathcal N_{\Q_T}$-a.e.. \medskip
   
 \noindent (ii) The vector field  $\gamma_x$ belongs to  $ D(\mathcal M)\cap D(\mathcal M^*)$ and is It\^o integrable (we denote by $I(\gamma_x)$ the It\^o's integral).\medskip
        
 \noindent     (iii)   $\mbox{\rm det}_2\,(I-\mathcal M(\gamma_x)(h))=1$ for   $h\in \X$,  $\mathcal N_{\overline{\Q}_T}$-a.e.\nor \footnote{$\det_2$ represents the determinant of Carleman-Fredholm.}
\medskip
 
 \noindent (iv) The law of $Z_x$ on  $E_{[0,T]}$ is given,  for  all $\Phi\in  B_b(E_{[0,T]})$, by
 $$
(\P\circ Z_x^{-1})(\Phi) =\int_{ E_{[0,T]}\nor }\Phi(h+e^{\cdot A}x)\, \exp\left\{ -\tfrac12|\gamma_x(h)|^2_{\mathcal H_{\Q_T}}  + I(\gamma_x)(h)\right\}\mathcal N_{\Q_T}(dh). $$
 \end{Theorem}
\begin{proof}
 We start from the identity
 $$
(\P\circ Z_x^{-1})(\Phi)= 
\int_{E_{[0,T]}}\Phi(h+e^{\cdot A}x)\,(\mathcal N_{\Q_T}\circ G_x)(dh),\quad \forall\,\Phi\in  B_b(E_{[0,T]})
 $$
 and apply    Proposition \ref{p3.1}, with $\Y=E_{[0,T]}$,    $\Lambda=F_x$, $\phi=\gamma_x$ given by \eqref{e3.2} , $\U=\Q_T $ and moreover
  proving that $\mbox{\rm det}_2\,(I-\mathcal M(\gamma_x)](h))=1$.
  
  We proceed in different steps.\medskip
  
  {\it Step 1}. 
     $ \gamma_x(h)\in \mathcal H_{\Q_T}, a.e.$ 

 
In fact, it results
$$
 \gamma_x(h)=e^{\cdot A}*b(h+e^{\cdot  A}x), \quad a.e.\,h\in\X. \nor 
$$
So, taking into account \eqref{e2.8} yields
$$
\|u\|^2_{\mathcal H_{\overline{\Q}_T}}=|u|^2_{\Gamma_{A}(H)}+|(-A)^{1/2}u(T)|_H 
$$
 Now   Step 1 follows from \eqref{eA5} and \eqref{eA7}  using the Lipschitz assumption on $b$. \nor
    \medskip

  {\it Step 2}.  $\gamma_x$ belongs to  $D(\mathcal M^*) $ and it is It\^o integrable.
  
   First note  that $\gamma_x$ belongs to  $D(\mathcal M)$.
In fact, due to  Hypothesis \ref{h1}, the mapping
 $$
h\in   E_{[0,T]} \nor \to \gamma_x(h)\in  E_{[0,T]} ,   $$
is Lipschitz, so that 
 $\gamma_x$  belongs to the domain of $\mathcal M$ by a well known result,
 see e.g., \cite{Bo98}. Now    it follows that  $\gamma_x\in D(\mathcal M^*)$ for all $x\in H$, see \cite[pag. 240]{Nu06} or \cite{Bo98},  \blu where $\mathcal M^*$ is indicated by $-\delta$\nor. Moreover, as we notice, $\gamma_x$ is adapted to the natural filtration of the Brownian motion, so that it is It\^o integrable. 
 Step 2 is proved.\medskip 
 
 To prove the last step let  us  first recall that if $T\in \mathcal L(\X)$ is Hilbert--Schmidt  with real eigenvalues   $(k_n)$, then $$\mbox{\rm det}_2(I-T)=\prod_{n=1}^\infty (1-k_n) e^{-k_n}.$$
  	If, in addition, $T$ is  {\em quasi--nilpotent}, it results  
   $$\mbox{\rm det}_2(I-T)=1,$$
   because in this case all $k_n$ are zero.

 {\it Step 3}. $\mathcal M(\gamma_x)(k)$ is {\em quasi--nilpotent} for any  $k\in \X$, so that 
 $$\mbox{\rm det}_2(I-[\mathcal M(\gamma_x)]  (h) ) =1, \nor \quad h\in\X.$$
  Assume for a moment that  $b$ is   $C^1$. Then we have 
$$
[\mathcal M \nor(\gamma_x)(k)\cdot h](\cdot)=(\overline{\Q}_T)^{1/2}\int_0^\cdot e^{(\cdot-s)A}b^\prime(k(r)+e^{rA}x)\cdot h(r)\,dr.
 $$
 It follows by recurrence that,
$$
 \| \mathcal M \nor  (\gamma_x)(k)\|_{\mathcal L(\X)}^n\le \tfrac1{n!}\,\|\overline{\Q}_T\|_{\mathcal L(\X)}^{n/2}\,\|b^\prime\|^n_\infty T^n,
$$
so that, from the inequality $\|(\mathcal M(\gamma_x)(k))^n\|_{\mathcal L(\X)}\le \| \mathcal M  (\gamma_x)(k)\|_{\mathcal L(\X)}^n$, we have
 $$
\lim_{n\to\infty}\,\|(\mathcal M(\gamma_x)(k))^n\|_{\mathcal L(\X)}^{1/n}=0,\, \quad\forall\,h\in \X,\,x\in H,
$$
which implies that  $\mathcal M(\gamma_x)(k)$ is {\em quasi--nilpotent} for all  $k\in \X$, as required.
If  $b$ is not  $C^1$ we proceed by approximation.
 
The proof  of the theorem is complete.
\end{proof} 
  Let $P_t,\;t\ge 0,$ be the transition semigroup corresponding to the process $Z_x$,
$$
P_t\varphi(x)=\E[\varphi(Z_x(t))],\quad \varphi\in B_b(H).
$$

\blu From (iv) in Theorem \ref{t3.1}
it follows \nor 
 \begin{Corollary}
\label{c3.3}
It results
\begin{equation}
  \label{e3.3}
P_T\varphi(x)= \E[\varphi(Z_x(T))]=\int_\X\varphi(k(T)+e^{TA}x)\,\exp\left\{ -\tfrac12|\gamma_x(k)|^2_{\mathcal H_{\Q_T}}  + [I(\gamma_x)](k)\right\}\,\mathcal N_{\Q_T} (dk),\quad \varphi\in B_b(H).
 \end{equation} 
 \end{Corollary}

 \subsection{Some estimates}
 
 Here we will not distinguish between  $\mathcal N_{\Q_T}$ and $\mathcal N_{\overline{\Q}_T}$. Moreover, \nor
   in this section we assume, besides Hypothesis \ref{h1}, that $b$ is bounded.
Obviously, recalling \eqref{ei1.11}, $\varrho(x,k)\in L^1(\X,\mathcal N_{\Q_T}$) for all  $x\in H$, but we do not  know whether 
$$
\int_\X\varrho^n(x,k)\,\mathcal N_{\Q_T}(dk)<\infty,\quad n\ge 2,
$$
or not.
An estimate for $\int_\X\varrho^n(x,h)\,\mathcal N_{\Q_T}(dh)$ is provided by the following result.
 \begin{Proposition}
\label{p3.9}
 Assume  Hypothesis \ref{h1}  with $b$ bounded and let $x\in H$ and $n>1. $Then it results
\begin{equation}
\label{e3.16}
\int_\X\varrho^n(x,h)\,N_{\Q_T}(dh)\le  \exp\left\{(n^2-n)\,\|b\|^2_\infty\, T\right\}.
\end{equation}

\end{Proposition}
\begin{proof}
Obviously we have
\begin{equation}
\label{e3.17}
\int_\X\varrho(x,h)\,\mathcal N_{\Q_T}(dh)=1.
\end{equation}
To estimate  $\varrho^n$ we write
  
$$\varrho^n(x,h)=\exp\left\{-\tfrac{n}2| \gamma_x(h)|^2_{\mathcal H_{\Q_T}} +n[I(\gamma_x)](h)\right\}.$$
Then setting
$$\widetilde{\varrho}(x,h)=\exp\left\{-\tfrac{n^2}2|\gamma_x(h)|^2_{\mathcal H_{\Q_T}}  +n[ I(\gamma_x)](h)\right\},   
$$
we see that  $\widetilde\varrho$ is obtained by $\varrho$ replacing $b$ with $nb$, so that
\begin{equation}
\label{e3.18}
\int_\X\widetilde\varrho(x,h)\,\mathcal N_{\Q_T}(dh)=1.
\end{equation} 
Therefore, taking into account that
$$\varrho^n(x,h)=\exp\left\{\tfrac{n^2-n}2\, |\gamma_x(h)|^2_{\mathcal H_{\Q_T}} \right\} \,\widetilde\varrho(x,h),$$
 \nor   and that  
     \begin{equation}
\label{e3.19} 
|\gamma_x(h)|^2_{\mathcal H_{\Q_T}} = |\Q_T^{-1/2}\gamma_x(h)|^2_\X \le  2|b(k(\cdot)+e^{\cdot A}x)|_\X^2\le  2\|b\|^2_\infty\,T,
 \end{equation}
  we find
$$\varrho^n(x,k)\le \exp\left\{(n^2-n)\,\|b\|^2_\infty\,T\right\} \,\widetilde\varrho(x,k).$$
Integrating with respect to $\X$ and taking into account \eqref{e3.18},
yields the conclusion.
\end{proof}

 \begin{Remark}
 \em
 By Proposition \ref{p3.9} and \eqref{e18aa} we can find an estimate of  $|I(\gamma_x)(h)|^2_{L^2(\X)}$ \nor as follows.
Write
 \begin{equation}
\label{e3.321}
 \begin{array}{l}
\ds  
  |I(\gamma_x)|^2_{L^2(\X)}   \nor 
=\int_{\X} \left|\int_{0}^T\langle[\gamma_{x}(h)](t),dB(t)\rangle_H\right|^2  \mathcal N_{\Q_T}(dh) \nor =  \int_{\X} \int_{0}^T\left|[\gamma_{x}(h)](t)\right|^2_H\,dt \mathcal N_{\Q_T}(dh) \nor 
\\
\\
\ds=   \int_{\X} \int_{0}^T\left|\int_{0}^te^{(t-r)A}\,b(h(r)+e^{rA}x)dr\right|^2_H\,dt  \mathcal N_{\Q_T} (dh)   \nor \\
\\
\ds\le \|b\|_\infty^2\int_{0}^T\int_{0}^t\,e^{-2\omega r}\, dr\,dt=\frac1{2\omega}\|b\|_\infty^2\int_0^T(1-e^{-\omega t})dt\le
\frac{T}{2\omega}\|b\|_\infty^2
\end{array}
\end{equation}

 \end{Remark}

   \section{The case when $b$ dissipative}

   \subsection{Preliminaries on  Gaussian measures on locally 
 convex spaces}
  We follow here \cite{Bo98}.  We are given a complete, separable,  locally  convex space $E$.
 We denote by $E^*$ the topological dual of $E$ and by $\mathcal B(E)$ the $\sigma$-algebra of all Borel subsets of $E$.

 A probability measure  $\mu$ on $(E,\mathcal B(E))$ is {\em Gaussian} if and only for any $F\in E^*$ the law  $\mu\circ F^{-1}$ of $F$ is Gaussian on $\R$.
 
Let $\mu$ be a Gaussian measure on $(E,\mathcal B(E))$. Then the {\em mean} of $\mu$ is  defined    as the linear functional $m:E^*\to \R$ given by,
\begin{equation}
\label{es1}
m(F)=  \int_E F(h)\mu(dh),\quad F\in E^*.
\end{equation}
The {\em covariance}  of $\mu$ is the mapping $Q\in \mathcal L(E^*, E^{**})$ defined as
$$
Q(F)=\int_E(F(h)-m(F))(h-m)\,\mu(dh),\quad F \in E^*.
$$

It follows that 
\begin{equation}
\label{es2}
 G(Q(F))=\int_E(F(h)-m(F))(G(h)-m(G))\,\mu(dh),\quad F,\,G\in E^*.
\end{equation}
  In particular, we have  
 \begin{equation}
\label{es3}
 F(Q(F))=\int_E(F(h)-m(F))^2\,\mu(dh),\quad F\in E^*.
\end{equation}
We note that the definitions of  $m$ and  $Q$ are meaningful thanks to the Fernique Theorem, see \cite[Theorem 2.8.5]{Bo98}.

  We denote by $\mathcal N_{m,Q}$ the Gaussian measure with mean $m$ and covariance $Q$.  If $m=0$, $\mu$ is said to be {\em symmetric} and is denoted by  $\mathcal N_{Q}$. All Gaussian measures considered in what follows are symmetric. In this case it results
 \begin{equation}
 \label{e0}
 F(Q(F))=\int_E(F(h) )^2\,\mu(dh),\quad F\in E^*
 \end{equation}
and
$$
  G(Q(F))=\int_E F(h)\,G(h )\,\mu(dh),\quad F,\,G\in E^*.
  $$
    $E^*$, endowed with the inner product
    $$(F,G)_\mu=G(Q(F)), \quad F,G\in E^*,$$
     is a pre--Hilbert space. We denote by $E^*_\mu$ its completion.
    
      Note that $Q$ is extendible  to $E^*_\mu$, setting. 
  $$
  Q(\psi)=\int_E \psi(k)\,k\,\mu(dk),\quad \psi\in  L^2(E, \mu).
  $$
 
 \medskip 
 
 We define    the {\em Cameron-Martin space} $\mathcal H_\mu$ of $\mu$  following \cite[pag. 44]{Bo98}
\begin{equation}
\label{es4}
 \mathcal H_\mu =\left\{h\in E:  \sup\Big\{F(h) : F\in E^*, \;\int_E |F(k)|^2\,\mu(dk)\le 1\Big\}<+\infty\right\}.
 \end{equation}
 If $h\in \mathcal H_\mu$ we set   following \cite[pag. 44]{Bo98}
 \begin{equation}
\label{es5}
|h|_{\mathcal H(\mu)}= 
\sup\Big\{F(h) : F\in E^*, \;\int_E |F(k)|^2\,\mu(dk)\le 1\Big\}.
 \end{equation}
It useful to notice that $h\in \mathcal H_\mu$ if and only if there exists  $\psi\in E^*_\mu$  such that  $h=Q(\psi)$.
 See \cite[Lemma 2.4.1]{Bo98}

 \subsection{Stationary process}   
 
According to Section 11.2 in \cite{DaZa02} we shall assume that
 \begin{Hypothesis}
\label{h2}
(i) Hypothesis \ref{h1} holds.\medskip

 (ii) $b\colon H\to H$ is dissipative, i.e., $\langle b(x)-b(y), x-y \rangle_H \le 0,$ for any $x,y \in H$.
 \end{Hypothesis}

 We start from  problem \eqref{e1.1}
whose solution we denote by $Z_x$ whereas the corresponding transition semigroup will be denoted by
$$
P_t\varphi(x)=\E[\varphi(Z_x(t))], \quad \varphi\in B_b(H),\,t\ge 0.
$$
It is convenient  to   modify problem  \eqref{e1.1}
by taking into account   a generic initial time $-n$
with $n\in\N.$
So, we consider the problem
\begin{equation}
\label{e5.1}
\left\{\begin{array}{l}
dZ(t)=(AZ(t)+b(Z(t))dt+dW(t),\quad t\ge -n,\\
\\
Z(x,-n)=x\in H,
\end{array}\right.
\end{equation}
where  $W$  is  an $H$--valued  cylindrical Wiener process defined for all  $t\in \R$ in the usual way.

Under Hypothesis 1 and the dissipativity of $b$, there is  a unique  solution  $Z_{x,-n}$ of the  mild equation
  \begin{equation}
\label{e5.2}
Z_{x,-n}(t)=e^{(t+n)A}x+\int_{-n}^te^{(t-r)A}\,b(Z_{x,-n}(r))dr+ W_{A,-n}(t),\quad t>-n,
\end{equation}
where   $ W_{A,-n}(t)$ denotes the modified  {\em stochastic convolution}
$$
 W_{A,-n}(t)=\int_{-n}^t e^{(t-r)A}\,dW(r).
$$
Moreover, there exists the limit
$$
Z_{\R}(t):=\lim_{n\to \infty}\,Z_{x,-n}(t),\quad \forall\, t\in\R
$$
and
we know by \cite{DaZa14} that $Z_{\R}$ is the unique in law, stationary solution of problem (1.1).  As it is well known, for all $t\in\R$ the law of $Z_{\R}(t)$ coincides with
    the unique invariant measure of $Z_x$ that we denote by $\nu$.

The stationary process $Z_{\R}$ is the solution to the mild limit equation
\begin{equation}
\label{e5.3}
Z_{\R}(t)=\int_{-\infty}^te^{(t-r)A}\,b(Z_{\R}(r))\,dr+  W_{A,\R}(t),\quad t\in \R
\end{equation}
where   $  W_{A,\R}$ denotes the modified  {\em stochastic convolution}
 $$
 W_{A,\R}(t)=\int_{-\infty}^t e^{(t-r)A}\,dW(r),\quad t\in\R.
$$
Clearly,  $ W_{A, \R}$ does not  live in $C_b(\R;H)$ but in $E_{\R}:=C(\R;H)$, the space of all continuous functions $\R\to H$. $E_{\R}$ is a separable  Fr\'echet space   endowed with   the set of  seminorms
  $$
  p_{j}(h)=\sup_{t\in [-j,j]}|h(t)|,\quad j\in\N,\;h\in E_\R.
  $$

 Let us also consider  the deterministic equation
\begin{equation}
\label{e5.4}
k(t)=\int_{-\infty}^t e^{(t-s)A}b(k(s))ds+h(t),\quad  t\in\R,\,h\in E_{\R},
\end{equation}
which has a unique solution $k_{\R}$  by a standard fixed point argument as in \eqref{e1.5}.  We have
 $k_{\R}=F_{\R}(h)$, where
 $$
F_{\R}\colon E_{\R}\to E_{\R},\quad h\to k_{\R}.
 $$
 $F_{\R}$ is a homeomorphism of $E_{\R}$ onto itself. We denote by $G_{\R}$ its inverse, so that
$$
G_{\R}(h)=h-\int_{-\infty}^\cdot e^{(t-s)A}b(h(s))ds.
$$
Finally, we define 
  \begin{equation}
\label{e5.5}
 [\gamma_{\R}(k)](t)=\int_{-\infty}^te^{(t-r)A}b(k(r))\,dr,\quad k\in E_{\R},\, t\in \R.
 \end{equation} 
 Therefore the solution of  \eqref{e5.3}  is given by
$$
Z_{\R}=F_{\R}(W_{A,\R}).
$$

We now consider   the law   $\mathcal N_{\Q_{\R}}$ of $W_{A,\R}$  on  $E_{\R}$  which is Gaussian  with mean zero and covariance   $\Q_{\R}$, given in Proposition \ref{p5.4} below. So, for all $\Phi\colon E_{\R}\to \R$  bounded and Borel we have
 $$
 (\mathcal N_{\Q_{\R}}\circ W_{A,\R}^{-1})(\Phi)= \E[ \Phi(W_{A,\R})]=\int_{E_{\R}}\Phi(h)\,\mathcal N_{\Q_{\R}}(dh).
 $$
 Before proving Proposition \ref{p5.4},
  we  recall  that the dual $E^*_{\R}$ of $E_{\R}$
    coincides with the space of all functions from $BV(\R;H)$  with a compact support.
If $F\in BV(\R;H)$ we shall write
 $$F(h)=\int_{-\infty}^{+\infty}\langle h(t),dF(t)\rangle_H,\quad \forall\, h\in E_{\R}.
 $$
 where the integral is intended in the sense of Stieltjes.
 We are now ready to show
  \begin{Proposition}
\label{p5.4}
 The law of  $W_{A,\R}$ in $E_{\R}$   is   the Gaussian  measure with mean $0$ and covariance $\Q_{\R}$
 given by
 \begin{equation}
\label{e5.7}
 F( \Q_{\R}(F))  =\frac1{2}  \int_{-\infty}^{+\infty}
 \int_{-\infty}^{+\infty}\langle(-A)^{-1}e^{|t-t_  1|A}dF(t),dF(t_1)\rangle_H  ,\quad \forall\,F\in  E^*_{\R}.
\end{equation}
Moreover, the corresponding Cameron--Martin space, denoted by $\mathcal H_{\Q_{\R}}$, is given by
\begin{equation}
\label{e5.77}
  \mathcal H_{\Q_{\R}}=
L^2(\R;D(A))\cap  W^{1,2}(\R;H).$$
\end{equation}
Finally,   if $u\in \mathcal H_{\Q_{\R}}, $
we have
 \begin{equation}
\label{e2.88}
 |u|^2_{\mathcal H_{\Q_{\R}}}=  \int_{-\infty}^{+\infty} (| u^\prime(t)|_H^2+| Au(t)|_H^2)\,dt.
\end{equation}
    \end{Proposition}
    \begin{proof}
    Let $F\in E^*_{\R}$: then by equation \eqref{e0} it follows that 
 $$
 F(\Q_{\R}(F))=\int_{\E_{\R}}|F(h)|^2\,\mathcal N_{\Q_{\\R}}(dh)=\E [|F(W_{A,\R)}|^2].
 $$
 Now, setting
  $$ F(h)=\int_{-\infty}^{+\infty}\langle h(t),dF(t)\rangle_H,\quad \forall\, h\in E_{\R},
 $$
 equation \eqref{e5.7} follows by direct computation of the variance of $W_{A,\R}$, equations \eqref{e5.77} and \eqref{e2.88} follow by analogous computations as in the proof of Lemma \ref{l2.1} and Proposition \ref{p2.2}.
 \end{proof}
  
 Now, we shall proceed as  in Theorem 3.1 proving, with the help of  the Ramer identity, that
$$
\mathcal N_{\Q_{\R}}\circ G_{\R} \ll \mathcal N_{\Q_{\R}},
$$
where  $G_\R= F^{-1}_{\R}$
and there exists $\varrho_{\R}\in L^1(E_{\R},\mathcal N_{\Q_{\R}}  )$  such that
\begin{equation}
\label{e5.100}
\varrho_{\R}(h)=\exp\left\{-\tfrac12|\gamma_{\R}(h)|^2_{\mathcal H_{\Q_{\R}}} + [\mathcal M_{\R}^*(\gamma_{\R})](h)\right\}.
  \end{equation}
 Here $\gamma_{\R}$ is defined by \eqref{e5.5} and $-\mathcal M^*_{\R}$ is the Skorokhod integral.
 Then we obtain
   \begin{Theorem}
\label{t4.3}
   Assume that Hypothesis \ref{h2} is fulfilled. Then the  law of  the stationary process $Z_{\R}$  in $E_{\R}$  is given by
    \begin{equation} 
  \label{e5.24}
 [\P\circ Z_{\R}^{-1}](\Phi)=  \E[\Phi(Z_{\R})]=\int_{\E_{\R}}\Phi(h)\, 
   \varrho_{\R}(h)\,\mathcal N_{\Q_{\R}}(dh),\quad \forall \,\Phi\in B_b(E_{\R}),
   \end{equation}
  where $\gamma_{\R}$ is defined by \eqref{e5.5} and $\varrho_{\R}$ by   \eqref{e5.100}.
  \end{Theorem}
  \begin{Remark}
  \em
  \label{r0}
  The reversed stationary process
 is given by
  $$
  \overline{Z}_{\R}(t)=Z_{\R}(-t),\quad t\in \R.
  $$
   \end{Remark}

\subsection{Invariant measures}   
 
   Assume  first  that $b=0$ and consider  the Ornstein--Uhlenbeck process
     \begin{equation} 
  \label{e5.24b}
  Y_{x}(t)=e^{tA}x+\int_0^t e^{(t-s)A}\,dW(s),\quad t\ge 0,
   \end{equation}
 and  its corresponding  stationary process  $Y_{\R}$,  
    \begin{equation} 
  \label{e5.24c}
 Y_{\R}(t):= \int_{-\infty}^t e^{(t-s)A}\,dW(s),\quad t\in\R.
   \end{equation}
 For any $t\in \R$ the   invariant measure of $Y_{\R}(t)$ is $\mu=:N_{1/2(-A)^{-1}}$.

   Let  moreover $P_t\varphi(x)= \E[ \varphi(Z_x(t))]$,  $R_t\varphi(x)= \E[\varphi( Y_x(t))]$ for all $x\in H,\,t\ge 0,\,\varphi\in B_b(H)$.

     \begin{Corollary}
 It results
  \begin{equation} 
  \label{e5.2423}
 [\P\circ Y_{\R}^{-1}](\Phi)= \int_{E_{\R}}\Phi(h)\, 
   \mathcal N_{\Q_{\R}}(dh),\quad \Phi\in B_b(E_{\R})
   \end{equation}
   and
    $\P\circ Z_{\R}^{-1}\ll \P\circ Y_{\R}^{-1}$.
\end{Corollary}
  \begin{proof}
Equation  \eqref{e5.2423} follows   by Theorem \ref{t4.3}, setting  $b=0$.
 Then 
the last statement follows  from comparing \eqref{e5.2423} with \eqref{e5.24}.
  \end{proof}
   \begin{Theorem}
\label{t3.3z}
Assume that Hypothesis \ref{h2} is fulfilled. 
Let $\nu$ and $\mu$ be  the invariant measures of $P_t$ and $R_t$ respectively.Then it results
 \begin{equation}
  \label{e3.3b}
\nu(\varphi)= \int_{E_{\R}}\varphi(k(0))\exp\left\{-\tfrac12|\gamma_{\R}(k)|^2_{\mathcal H_{\Q_{\R}}} + [\mathcal M_{\R}^*(\gamma_{\R})](k)\right\}\,  \mathcal N_{\Q_{\R}} (dk),\quad \forall\,\varphi\in B_b(H).
  \end{equation}
Moreover,
 \begin{equation}
  \label{e3.3c}
\mu(\varphi)= \int_{E_{\R}}\varphi(k(0))\, \mathcal N_{\Q_{\R}} (dk),\quad \forall\,\varphi\in B_b(H).
 \end{equation} 
 Consequently 
  \begin{equation}
  \label{e4.22}
   \nu \ll \mu.
  \end{equation} 
 \end{Theorem}
 \begin{proof}
 Setting  in \eqref{e5.24} $\Phi(h)=\varphi(h(0))$ yields \eqref{e3.3b}. The other statements are straightforward.
 \end{proof}
  A final result, yields a formula  for the density $\tfrac{d\nu}{d\mu}$.
   \begin{Proposition}
   \label{p4.6}  Let $E= E_{\R}.$
 Let  $p\colon E \to H$ be defined by\; $p(h) = h(0),\;  h \in E$ with $p^{-1}(x)=\{ h\in E:\,h(0)=x\}, \;\; x \in H$, setting 
\begin{gather*}
 \psi (x) = \int_{p^{-1}(x)} \varrho_{\R}(h)m_x(dh) \;=\E\big[\varrho_\R(\cdot)\mid k(0)=x\big] 
\end{gather*} 
we find that $\psi = \tfrac{d\nu}{d\mu}$, $\mu$-a.s.
 \end{Proposition}
 \begin{proof}
 We use a  disintegration argument, see Appendix B. Apply
 Theorem \ref{tA.1}   with  $\lambda =N_{\Q_{\R}}$, $E=E_{\R}$. 
 By \eqref{e4.22} we know that $\mu = \lambda \circ p^{-1}$ and 
 that there exists a family of  Borel measures $(m_x)_{x\in H}$ 
in $(E,\mathcal B(E))$  such that the support of $m_x$
is included in $p^{-1}(x)$, for $\mu$-almost all  $x\in H$, and
\begin{gather*}
\int_H \varphi(x) \,\nu(dx) = \int_E \varphi(h(0))\, \varrho_{\R}(h)\lambda(dh)
= \int_H \varphi(x) \left(\int_{p^{-1}(x)} \varrho_{\R}(h)m_x(dh)   \right)\,\mu(dx),
\end{gather*}
 for any $\varphi : H \to \mathbb{R}$ Borel and bounded.  We could also proceed directly,
 observing that
 $$
 \nu(\varphi)= \int_{E_{\R}}\varphi(k(0))\varrho_\R(k)\mathcal N_{\Q_{\R}} (dk)=
 \int_H \varphi(x)\,\E\big[\varrho_\R(\cdot)\mid k(0)=x\big] \,\mu(dx)
 $$
  
 \end{proof}

\section{Colored noise}  

We are  here concerned with the following  stochastic differential equation on $H$,
\begin{equation}
\label{e9.1}
\left\{\begin{array}{l}
dV(t)=(AV(t)+b(V(t))dt+(-A)^{-\epsilon/2}\,dW(t),\quad t\ge 0\\
\\
V(0)=x\in H,
\end{array}\right.
\end{equation}
assuming,  the following
\begin{Hypothesis}
\label{h3}
(i) $\epsilon>0.$ \medskip 

(ii) Hypothesis \ref{h1}(i)-(iii)  are fulfilled.\medskip

(iii) $(-A)^\epsilon\,b$ is Lipschitz in $H$.
 \end{Hypothesis}
 Proceeding as  in Section 1,  we see that problem \eqref{e9.1} has a unique mild solution  $V_x$  which is   a continuous  adapted process,
\begin{equation}
\label{e9.2}
V_x(t)=e^{tA}x+\int_0^te^{(t-s)A}\,b(V_x(s))ds+(-A)^{-\epsilon/2}\,W_A(t),\quad t\ge 0,
\end{equation}
 where  $W_A$ still  denote    the {\em stochastic convolution}
$$
  W_A(t)=\int_0^te^{(t-s)A}\, dW(s),\quad t\ge 0.
$$
 We again  fix  $T>0$ and set $E_{[0,T]}=C([0,T];H)$.
As we have seen in Section 2, the law of $W_A$ in $E_{[0,T]}$ is Gaussian $\mathcal N_{\Q_T}$, described on  Section 2 above.\medskip
 
   We are going to determine the 
law of $V_x$ on $E_{[0,T]}$ for all $x\in H$, by proceeding as in Section  1. Setting $V_x(t)-e^{tA}x=L_x(t),$ equation \eqref{e9.2} becomes
 \begin{equation}
\label{e9.3}
L_x(t)=\int_0^te^{(t-s)A}\, b(L_x(s)+e^{sA}x)ds+ (-A)^{-\epsilon/2}\,W_A(t),\quad t\in[0,T].
\end{equation}
 Moreover,  for every $x\in H$ we consider some  mappings defined in \S1 as $k_x,$ $F_x, G_x, \gamma_x$,
  see \eqref{e1.5}, \eqref{e1.6}, \eqref{e1.7} and \eqref{e1.8} respectively.

    Now the solution $L_x(\cdot)$ of \eqref{e9.3} is given by
 \begin{equation}
\label{e9.5}
L_x(\cdot)=F_x((-A)^{-\epsilon/2}\,W_A(\cdot))
\end{equation}
Therefore we have
\begin{equation}
\label{e9.6}
(\P\circ V_x^{-1})(\Phi)=\E [
\Phi(F_x( (-A)^{-\epsilon/2}W_A(\cdot))+e^{\cdot A}x)],\quad \forall\,\Phi\in  B_b(E_{[0,T]}).
\end{equation}
  By the  change of variables $ \Omega\to \X,\;\omega\to W_A(\cdot)(\omega)$ we obtain
\begin{equation}
  \label{e9.7}
 (\P\circ V_x^{-1})(\Phi)=\int_\X\Phi\left[F_x((-A)^{-\epsilon/2}\,h(\cdot))+e^{\cdot A}x\right]N_{\Q_T}(dh),\quad \Phi\in B_b(E_{[0,T]}),\,t\ge 0,
 \end{equation} 
 where $\mathcal N_{\Q_T}$ is the law of $W_A(\cdot)$  in $E_{[0,T]}$.

 Setting  $(-A)^{-\epsilon/2}h=k,$
yields
\begin{equation}
  \label{e9.8}
 (\P\circ V_x^{-1})(\Phi)=\int_\X\varphi\left[F_x(k(\cdot))+e^{\cdot A}x\right]\,\mathcal N_{ \overline{\Q}^\epsilon_T}(dk),\quad \varphi\in B_b(E_{[0,T]}),
 \end{equation} 
 where  $\overline{\Q}^\epsilon_T=(-A)^{-\epsilon}\overline{\Q}_T$. Note that $\overline{\Q}^\epsilon_T$  is obviously  symmetric and of trace class.
 
 Moreover, by the change of variables
$
F_x(h)=k,
$
we obtain
\begin{equation}
\label{e9.9}
(\P\circ V_x^{-1}))(\Phi)= 
\int_\X\Phi(k(\cdot)+e^{\cdot A}x)\,(\mathcal N_{ \overline{\Q}^\epsilon_T}\circ G_x)(dk),\quad \forall\,\Phi\in   B_b(E_{[0,T]}). \nor  
\end{equation}
We shall denote by $ \mathcal H^\epsilon_T = \mathcal H_{\overline{\Q}^\epsilon_T}$  
the Cameron--Martin space  of  $\mathcal N_{\overline{\Q}^\epsilon_T}$, endowed with its natural norm $|h|_{\mathcal H^\epsilon_T}.$
 Arguing as in \S 2 we see that  it is given by
 \begin{equation}
 \label{e9.10}
 \mathcal H^\epsilon_T:=  L^2(0,T;D((-A)^{1+\epsilon/2})\cap  W_0^{1,2}(0,T;D((-A)^{\epsilon/2})),
 \end{equation}
 Moreover
 \begin{equation}
\label{e2.98}
 |u|^2_{\mathcal H^\epsilon_T}:=|\overline{Q^\epsilon_T}^{-1/2}u|^2_\X= |(-A)^{1/2}\,u(T)|_{D((-A)^{\epsilon/2})}^2 +\int_0^T ( | u^\prime(t)|_{D{((-A)^{\epsilon/2})}}^2  +| Au(t)|_{D((-A)^{\epsilon/2})})\,dt.
\end{equation}

 Hypothesis \ref{h3}  ensures that $\gamma_x(k)\in \mathcal H^\epsilon_T$ 
 because 
$$
b(h+e^{\cdot A}x) \in L^2(0,T, D((-A)^{\epsilon/2}). 
$$
Now, by  proceeding  as in Section 3 we prove the following result. 
  \begin{Theorem}
\label{t9.3}
The  law of $V_x$ on $E_{[0,T]}$ is given by
\begin{equation}
\label{e9.12}
(\P\circ V_x^{-1})(\Phi)=\int_\X\Phi(h+e^{\cdot A}x),\varrho_\epsilon(x,k)\,\mathcal N_{\overline{\Q}^\epsilon_T}(dk),\quad \forall \, \Phi\in   B_b(E_{[0,T]}), \nor 
\end{equation}
where
\begin{equation}
\label{e9.13}
\varrho_\epsilon(x,k)=\exp\left\{-\tfrac12|\gamma_x(k)|^2_{\mathcal H^\epsilon_T} +I(\gamma_x)(k)\right\}
\end{equation}
 and
   \begin{equation}
\label{e9.14}
 \gamma_x(k)= \int_0^\cdot e^{(\cdot-s)A}b(k(s)+e^{sA}x)\,ds,\quad x\in H,\,k\in E_{[0,T]}.
 \end{equation}
 \end{Theorem}
 \begin{Remark}
 \em Larger is  $\epsilon>0$ stronger becomes (iii) in Hypothesis \ref{h3}
  and moreover narrow is the Cameron--Martin space.
 The Hypothesis \ref{h3}(iii) becomes stronger and the Cameron-Martin space narrower as $\epsilon>0$ grows larger.
\end{Remark}
 \begin{Remark}
 \em Theorem \ref{t3.3z} and Proposition \ref{p4.6} can be easily generalised
 to the colored noise.
 \end{Remark}
 
 \appendix 
\addcontentsline{toc}{section}{Appendix}

 \section{Maximal regularity   for linear  evolution equations}
Let us consider  an abstract  evolution equation,
 \begin{equation}
\label{eA1}
u^\prime(t)=Au(t)+f(t),\quad u(0)=0,\quad t\in[0,T],
\end{equation}
on   a Hilbert space  $H$, where $A:D(A)\subset H\to H$ is self--adjoint negative and   $f\in L^2(0,T;H)$. 

The following result is well known, see e.g., \cite{Lu18}, we recall the  easy proof, however, for the reader convenience.
 \begin{Proposition}
\label{pA1}
Let
\begin{equation}
\label{eA2}
u(t)=\int_0^te^{(t-s)A}f(s)\,ds=(e^{\cdot A}*f)(t),\quad t\in[0,T].
\end{equation}
Then $u\in W^{1,2}(0,T,H)\cap L^2(0,T;D(A))$    and fulfils \eqref{eA1}. Moreover, it results
\begin{equation}
\label{eA3}
|u^\prime|_{L^2(0,T;H)}\le2 |f|_{L^2(0,T;H)},\quad|Au|_{L^2(0,T;H)}\le |f|_{L^2(0,T;H)}.
\end{equation}
\end{Proposition}
\begin{proof}
Denote by  $\hat f(k),\, k\in H,$ the Fourier transform of $f$  and by  $\hat u(k),\, k\in H,$ the Fourier transform of $u$.
($f$ and $u$ are extended by $0$ outside $[0,T]$).
Taking the Fourier transform on both sides of 
\eqref{eA2}  yields
$$
\hat u(k)=A(k-A)^{-1}\hat f(k),\quad k\in H.
$$
Since $\|A(k-A)^{-1}\|_{\mathcal L(H)} \le 1$ we have $|\hat u(k)|\le |\hat f(k)|$  for all $k\in H$
and the conclusion follows.
\end{proof}

Let us define the maximal regularity space
$\Gamma_{A}(H)$  by setting
 \begin{equation}
\label{eA4}
\Gamma_{A}(H)=L^2(0,T;D(A))\cap \{u\in W^{1,2}(0,T;H):\,u(0)=0\}.
\end{equation}
Then $\Gamma_{A}(H)$, endowed with the norm: 
$$
|u|^2_{\Gamma_{A}(H)}= \int_0^T\left( |u^\prime(t)|_H^2+|Au(t)|_H^2\right) \,dt,
$$
is a Hilbert space.

If $u=e^{\cdot A}*f$ we have
\begin{equation}
\label{eA5}
 |e^{\cdot A}*f|_{\Gamma_{A}(H)}\le c_1|f|_\X.
\end{equation}
Moreover, since it results
\begin{equation}
\label{eA6}
\Gamma_{A}(H)\subset  C([0,T;D((-A)^{1/2}),
\end{equation}
with continuous inclusion,
see \cite{BeDaDeMi06} and \cite{Lu18},
 there is $c_2>0$ such that
\begin{equation}
\label{eA7}
|((-A)^{1/2}u(T)|_H\le c_2\,|f|_\X.
\end{equation}

  \begin{Corollary}
\label{cA.2}
Let $\epsilon>0$, $f\in L^2(0,T;D((-A)^{\epsilon/2}))$. Then
 \begin{equation}
\label{eA.8}
u\in L^2(0,T;D((-A)^{1+\epsilon/2})\cap W^{1,2}(0,T;D((-A)^{\epsilon/2})).
\end{equation}
Moreover, the following continuous inclusion holds
\begin{equation}
\label{eA.9}
L^2(0,T;D((-A)^{1+\epsilon/2}))\cap W^{1,2}(0,T;D((-A)^{\epsilon/2}))\subset  C([0,T;(-A)^{\epsilon/2})).
\end{equation}
Finally,
\begin{equation}
\label{eA.10}
|Au|_{L^2(0,T;D((-A)^{\epsilon/2}))}\le |f|_{L^2(0,T;D((-A)^{\epsilon/2}))},\quad
|u^\prime|_{L^2(0,T;D((-A)^{\epsilon/2}))}\le |f|_{L^2(0,T;D((-A)^{\epsilon/2}))}.
\end{equation}
\end{Corollary}
\begin{proof}
It is sufficient to  apply 
Proposition \ref{pA1} replacing $H$ with $D((-A)^{\epsilon/2}).$
\end{proof}
 
 \section{A disintegration theorem}
  For the proof of the next result see, for instance, the appendix  in \cite{DLT}.\nor

%
%
%
 
\begin{Theorem}
\label{tA.1}
Let  $E$ be  a Polish space, $p\colon E\to H$ Borel, $\lambda\in \mathcal P(E)$ and $\mu= \lambda \circ p^{-1}$ the law of $p$. 
There exists a family of  Borel measures $(m_x)_{x\in H}$ 
in $(E,\mathcal B(E))$  such that
\begin{equation}
\label{eA.1}
\int_E \varphi(h)\lambda(dh)=\int_H \left(\int_{E} \varphi(h)m_x(dh)   \right)\,\mu(dx),
\end{equation}
for all $\varphi\colon E\to \R$ bounded and Borel.

Moreover the support of $m_x$
is included in $p^{-1}(x)$ for  $\mu$-almost all  $x\in H$, so that we can write \eqref{eA.1} as 
\begin{equation}
\label{eA.1c}
\int_{E} \varphi(h)\lambda(dh)=
 \int_H \left(\int_{p^{-1}(x)} \varphi(h)m_x(dh)   \right)\,\mu(dx).  
\end{equation}
\end{Theorem}

\section*{Acknowledgments}  
We thank the referee for his careful reading of the article
and for his precise observations.\\
Further comments and references will be posted on the webpage

https://tubaro.maths.unitn.it\\
under the heading ``L'ultimo lavoro di Beppe Da Prato''.\\
The authors  thank also  Vladimir I. Bogachev,  Arnaud Debussche and Giorgio Metafune
for useful discussions. Enrico Priola  is  a member of  GNAMPA (INDAM) and his research activity was carried out as part of the PRIN 2022 project “Noise in fluid dynamics and
related models”.


\end{document}